\documentclass[11pt,a4paper]{article}
\usepackage[utf8]{inputenc}
\usepackage[all]{xy}
\usepackage[margin=3cm]{geometry}
\usepackage[dvipdfmx]{graphicx}

\usepackage{authblk}
\usepackage{hyperref}

\usepackage{mathtools}
\usepackage{comment}
\usepackage{stmaryrd} % \varoast

\usepackage[
  backend=biber,
    style=ext-alphabetic,
    maxalphanames = 5,
    maxcitenames  = 3,
    maxbibnames=5,
    sorting=nyt,
    giveninits,
    articlein=false,
    url=false,  
    doi=false,
    isbn=false,
    eprint=true
]{biblatex}
\SelectTips{cm}{}

\usepackage{color}
\usepackage{subcaption}
\usepackage{enumerate}
\usepackage{array}
\usepackage{float}
\usepackage[dvipdfmx]{graphicx}

\usepackage{latexsym}
\usepackage{amscd}
\usepackage{amsmath}
\usepackage{amssymb}
\usepackage{amsthm}
\usepackage{mathrsfs}
\usepackage{cleveref}

\theoremstyle{plain}
\newtheorem{theorem}{Theorem}[section]
\crefname{theorem}{Theorem}{Theorems}

\crefname{corollary}{Corollary}{Corollaries}
\newtheorem{lemma}[theorem]{Lemma}
\crefname{lemma}{Lemma}{Lemmas}

\crefname{claim}{Claim}{Claims}

\newtheorem{proposition}[theorem]{Proposition}
\crefname{proposition}{Proposition}{Propositions}

\theoremstyle{definition}
\newtheorem{definition}[theorem]{Definition}
\crefname{definition}{Definition}{Definitions}
\newtheorem{remark}[theorem]{Remark}
\crefname{remark}{Remark}{Remarks}
\newtheorem{remarks}[theorem]{Remarks}

\crefname{example}{Example}{Examples}

\crefname{section}{Section}{Sections}
\crefname{subsection}{Subsection}{Subsections}
\crefname{equation}{Equation}{Equations}
\crefname{figure}{Figure}{Figures}

\numberwithin{equation}{section}
\numberwithin{figure}{section}
\numberwithin{table}{section}

\usepackage{tikz}
\newcommand{\id}{\operatorname{id}}

\newcommand{\holim}{\operatorname{holim}}
\DeclareMathOperator*{\hocolim}{\operatorname{hocolim}}

\newcommand{\Hom}{\operatorname{Hom}}
\newcommand{\End}{\operatorname{End}}
\newcommand{\cHom}{\mathop{{\cH}om}\nolimits}
\newcommand{\RHom}{\mathop{\mathrm{R}\mathrm{Hom}}\nolimits}
\newcommand{\cRHom}{\mathop{\mathrm{R}\mathcal{H}om}\nolimits}

\newcommand{\bN}{\ensuremath{\mathbb{N}}}
\newcommand{\bR}{\ensuremath{\mathbb{R}}}
\newcommand{\bZ}{\ensuremath{\mathbb{Z}}}

\newcommand{\scA}{\ensuremath{\mathcal{A}}}
\newcommand{\scC}{\ensuremath{\mathcal{C}}}
\newcommand{\scI}{\ensuremath{\mathcal{I}}}

\newcommand{\cD}{\ensuremath{\mathcal{D}}}
\newcommand{\cH}{\ensuremath{\mathcal{H}}}
\newcommand{\cK}{\ensuremath{\mathcal{K}}}

\newcommand{\cL}{\ensuremath{\mathcal{L}}}
\newcommand{\cT}{\ensuremath{\mathcal{T}}}

\newcommand{\frakL}{\ensuremath{\mathfrak{L}}}
\newcommand{\bfk}{\ensuremath{\mathbf{k}}}
\newcommand{\simto}{\xrightarrow{\sim}}

\newcommand{\pt}{\mathrm{pt}}
\newcommand{\SD}{\mathsf{D}}
\DeclareMathOperator{\Int}{\operatorname{Int}}
\DeclareMathOperator{\supp}{\operatorname{supp}}
\DeclareMathOperator{\MS}{{\operatorname{SS}}}
\DeclareMathOperator{\rMS}{{\mathring{\MS}}}

\newcommand{\rR}{\mathrm{R}}

\DeclareMathOperator{\gammasupp}{\gamma-supp}
\DeclareMathOperator{\RS}{RS}

\DeclareMathOperator{\DHam}{\operatorname{DHam}}
\newcommand{\cdis}{\mathrm{c}}
\newcommand{\Ch}{\mathrm{Ch}}
\newcommand{\cstar}{\ast}
\newcommand{\homst}{\cHom^*}
\newcommand{\convstar}{\varoast}
\newcommand{\RGamma}{\rR\Gamma}

\newcommand{\cTlc}{\cT_{\mathrm{lc}}}

\title{Regular Lagrangians are smooth Lagrangians
}
\author[1]{Tomohiro Asano\thanks{\texttt{tasano@kurims.kyoto-u.ac.jp}, \texttt{tomoh.asano@gmail.com}. Supported by JSPS KAKENHI Grant Number JP24K16920. Also supported by JST, CREST Grant Number JPMJCR24Q1, Japan.}}
\author[2]{St\'ephane Guillermou\thanks{\texttt{Stephane.Guillermou@univ-nantes.fr}. Supported by ANR COSY (ANR-21-CE40-0002) and Centre Henri Lebesgue  (ANR-11-LABX-0020-01).}}
\author[3]{ \\ Yuichi Ike\thanks{\texttt{ike@imi.kyushu-u.ac.jp}, \texttt{yuichi.ike.1990@gmail.com}. Supported by JSPS KAKENHI Grant Numbers JP21K13801 and JP22H05107. Also supported by JST, CREST Grant Number JPMJCR24Q1, Japan.}}
\author[4]{Claude Viterbo\thanks{\texttt{Claude.Viterbo@universite-paris-saclay.fr}. Supported by ANR COSY (ANR-21-CE40-0002).}}
\affil[1]{Research Institute for Mathematical Sciences, Kyoto University, Kitashirakawa-Oiwake-Cho, Sakyo-ku, 606-8502, Kyoto, Japan.}
\affil[2]{UMR CNRS 6629 du CNRS Laboratoire de Math\'ematiques Jean LERAY
2 Chemin de la Houssini\`ere, BP 92208, F-44322 NANTES Cedex 3 France}
\affil[3]{Institute of Mathematics for Industry, Kyushu University, 744 Motooka, Nishi-ku, Fukuoka-shi, Fukuoka 819-0395, Japan.}
\affil[4]{Universit\'e Paris-Saclay, CNRS, Laboratoire de Math\'ematiques d'Orsay, 91405, Orsay, France. }

\date{\today}
\DeclareTextFontCommand\textqt{\qtLondon}
\begin{document}

\maketitle

\centerline{To Pierre Schapira for his $80^{th}$ birthday.}

\begin{abstract}
  We prove that for any element in the $\gamma$-completion of the space of smooth compact exact Lagrangian submanifolds of a cotangent bundle, if its $\gamma$-support is a smooth Lagrangian submanifold, then the element itself is a smooth Lagrangian. We also prove that if the $\gamma$-support of an element in the completion is compact, then it is connected.
\end{abstract}

\section{Introduction}

Let $M$ be a $C^\infty$ closed connected manifold. 
The space $\frakL(T^*M)$ of smooth compact exact Lagrangian submanifolds of $T^*M$ carries a distance $\gamma$, called the spectral distance (see \cite{viterbo1992,Oh,M-V-Z,H-L-S}). 
The metric space $(\frakL(T^*M), \gamma)$ is not complete, so we consider its completion. 
Its study was initiated in \cite{Humiliere-completion}, pursued further in \cite{viterbo2022supports}, and has applications to Hamilton--Jacobi equations \cite{Humiliere-completion}, symplectic homogenization theory \cite{Viterbo-Homogenization}, and to conformally symplectic dynamics \cite{AHV2024higher}.

The elements of the completion $\widehat\frakL(T^*M)$ are by definition certain equivalence classes of Cauchy sequences with respect to the spectral distance $\gamma$. 
Despite their very abstract nature, they admit a geometric incarnation called the \emph{$\gamma$-support}, which was introduced by Viterbo in \cite{viterbo2022supports} (as a modification of the support introduced in \cite{Humiliere-completion}). 
It is defined as follows:
\begin{definition}
    Let $L_\infty \in \widehat\frakL(T^*M)$ and $z\in T^*M$. One says that $z$ is in the \emph{$\gamma$-support} of $L_\infty$ if for any neighborhood $U$ of $z$ there is $\varphi\in\DHam_c(U)$ such that $\varphi(L_\infty)\neq L_\infty$.
    Here, $\DHam_c(U)$ denotes the group of Hamiltonian diffeomorphisms compactly supported in $U$.
    The set of points in the $\gamma$-support of $L_\infty$ is denoted by $\gammasupp(L_\infty)$. 
\end{definition}

For a smooth Lagrangian $L\in\frakL(T^*M)$, we easily show $\gammasupp(L)=L$.
Several questions are of importance for $\gammasupp(L_\infty)$. 
Does $\gammasupp(L_\infty)$ characterize $L_\infty$? This is not the case in general (examples can be found in \cite{viterbo2022supports}), but one could still hope it if $\gammasupp(L_\infty)$ is small.

Also since $\gamma$-supports appear in \cite{AHV2024higher} as higher-dimensional versions of Birkhoff invariant sets, they share some of the properties of the $1$-dimensional case. 
It is proved in loc.~cit.\ that the projection $\pi \colon \gammasupp(L_\infty) \to M$ induces an injection in cohomology, but also that the map is not in general surjective. 
However, is it the case at the $H^0$ level? 

In this note, we give positive answers to the above questions, namely Conjecture~8.2 of \cite{viterbo2022supports} and a question in \cite{AHV2024higher}. That is, we prove, for $L_\infty \in \widehat\frakL(T^*M)$,
\begin{itemize}
\item [(i)] if $\gammasupp(L_\infty)=L$ for some $L \in \frakL(T^*M)$, then $L_\infty=L$ (see \cref{thm:main}),
\item [(ii)] if $\gammasupp(L_\infty)$ is compact, then it is connected (see \cref{thm:gammasupp_connected}).
\end{itemize}

\subsection*{Acknowledgments}
The authors are very grateful to Vincent Humili\`ere for many helpful discussions. This paper is the continuation of \cite{AGHIV}.
They also thank Bingyu Zhang for the careful reading of the manuscript.
They thank the anonymous referee for the helpful comments.

\section{Notations}

Throughout this paper, we fix a field $\bfk$.

Let $\cL(T^*M)$ denote the set of compact exact Lagrangian branes, i.e., triples $(L, f_L, \widetilde G)$, where $L$ is a compact exact Lagrangian submanifold of $T^*M$, $f_L \colon L \to \bR$ is a function satisfying
$df_L=\lambda_{\mid L}$, and $\widetilde G$ is a grading of $L$ (see \cite{Seidel-graded,viterbo2022supports}).  
The action of $\bR$ on $\cL(T^*M)$ given by $(L,f_L, \widetilde G) \mapsto (L,f_L-c, \widetilde G)$ is denoted by $T_c$. 
Let $\mathfrak L(T^*M)$ be the set of compact exact Lagrangians, where we do not record primitives or gradings.  
For $L_1, L_2$ in $\cL(T^*M)$, we define as in \cite{viterbo2022supports} the spectral invariants $c_+(L_1,L_2)$ and $c_-(L_1,L_2)$, and set  
\[
    \cdis (L_1,L_2)= \max\{c_+(L_1,L_2),0\}-\min\{c_-(L_1,L_2),0\}.
\]
This defines a distance.\footnote{Note that the definition given in \cite{AGHIV} is not correct, and has to be replaced by the one above. This has been corrected in the published version of \cite{viterbo2022supports}.} For $L_1, L_2$ in $\mathfrak L(T^*M)$, we define the spectral distance between $L_1$ and $L_2$ by
\[
    \gamma(L_1,L_2)= \inf_{c\in \mathbb R} \cdis (L_1,T_c L_2)=c_+(L_1,L_2)-c_-(L_1,L_2).
\]
We denote by $\widehat{\frakL}(T^*M)$ (resp.\ $\widehat{\cL}(T^*M)$) the completion of $\frakL(T^*M)$ (resp.\ $\cL(T^*M)$) with respect to $\gamma$ (resp.\ $\cdis$). 

We denote by $\DHam(T^*M)$ the group of Hamiltonian diffeomorphisms of $T^*M$ (time $1$ of an isotopy) and $\DHam_c(T^*M)$  its subgroup made by times 1 of compactly supported isotopies.

We follow the notations of \cite{KS90}. In particular $\SD(\bfk_M)$ is the derived category of sheaves of $\bfk$-vector spaces on $M$. An object $F\in \SD(\bfk_M)$ has a microsupport $\MS(F) \subset T^*M$ defined in loc.~cit. For $A \subset T^*M$, a closed conic subset, $\SD_{A}(\bfk_M) \coloneqq \{F \in \SD(\bfk_M) \mid \MS(F) \subset A\}$ is a triangulated full subcategory of $\SD(\bfk_M)$.  
We now recall several notions and ideas from~\cite{Tamarkin}.
We denote by $(t;\tau)$ the canonical coordinates on $T^*\bR$ and we set for short $\{ \tau \gtrless 0\} = T^*M \times \{ \tau \gtrless 0\} \subset T^*(M \times \bR_t)$.
The Tamarkin category $\cT(T^*M)$ is defined as the quotient category $\SD(\bfk_{M \times \bR})/\SD_{\{\tau \le 0\}}(\bfk_{M \times \bR})$.
The Tamarkin category has a monoidal structure. 
For $F,F' \in \SD(\bfk_{M \times \bR})$ we set $F \cstar F' \coloneqq \rR m_!(q_1^{-1}F \otimes q_{2}^{-1}F')$, where $q_1, q_2\colon M \times \bR^2 \to M\times\bR$ are the projections and $m$ is the addition map $m(x,s,t)=(x,s+t)$. 
The operation $\cstar$ preserves the left orthogonal ${}^\perp \SD_{\{\tau \le 0\}}(\bfk_{M \times \bR})$ and moreover $F \mapsto F \cstar \bfk_{M\times [0,+\infty[}$ is a projector onto it.  
This projector induces an equivalence between $\cT(T^*M)$ and ${}^\perp \SD_{\{\tau \le 0\}}(\bfk_{M \times \bR})$, with which we identify them in what follows.
We also set $\homst(F,F') \coloneqq \rR q_{1*} \cRHom(q_{2}^{-1}F, m^!F')$ and denote the projection of this $\homst$ onto $\cT(T^*M)$ by the same symbol.
This defines an internal hom $\homst \colon \cT(T^*M)^\mathrm{op} \times \cT(T^*M) \to \cT(T^*M)$.
For $c \in \bR$, let $T_c\colon M\times\bR \to M\times\bR$ be the translation $T_c(x,t) = (x,t+c)$.  
The category $\cT(T^*M)$ comes with a family of morphisms of functors $\tau_c\colon \id \to T_{c*}$ for each $c\geq 0$ introduced by Tamarkin. 
They give rise to an interleaving distance on $\cT(T^*M)$ denoted $d_{\cT(T^*M)}$
(see~\cite{K-S-distance} and~\cite{AI20}) defined as follows:
\[
    d_{\cT(T^*M)}(F,F') 
    \coloneqq 
    \inf\left\{ a+b \; \middle| \; 
    \begin{aligned}
        & \exists u \colon F\to T_{a*}F', \; \exists v \colon F' \to T_{b*}F, \\
        & T_{a*} v \circ u=\tau_{a+b}(F), \; T_{b*} u \circ v=\tau_{a+b}(F')     
    \end{aligned}
     \right\}.
\]

We recall the composition of sheaves.
For $F\in \SD(\bfk_{M \times N})$ and $G\in \SD(\bfk_{N \times P})$, set $F \circ G \coloneqq \rR q_{13!}( q_{12}^{-1}F \otimes q_{23}^{-1}G)$, where $q_{ij}$ are the projections from $M\times N\times P$ to the $(i\times j)$ factors.  
We also consider a mixture of $\circ$ and $\cstar$: for $F \in \cT(T^*M \times T^*N)$, $G \in \cT(T^*N \times T^*P)$, we set $F \convstar G = \rR m_!\rR q_{13!}( q_{12}^{-1}F \otimes q_{23}^{-1}G)$ where $q_{ij}$ are projections from $M\times N\times P \times \bR^2$ to $M\times N\times \bR$, $N\times P \times \bR$, $M\times P \times \bR^2$ and $m$ the addition map.  
We set for short $\cK^\convstar(F) \coloneqq \cK \convstar F$ for $\cK \in \cT(T^*M^2)$ and $F \in \cT(T^*M)$.

We put an analytic structure on $M$ and define $\cTlc(T^*M)$ as the subcategory of $\cT(T^*M)$ made by objects that are limits (for the interleaving distance) of constructible sheaves. We remark that for a submanifold $N$ of $M$, the pull-back
to $N\times\bR$ commutes with $T_{c*}$ and $\tau_c$. It follows that the pull-back is a contraction and hence
sends $\cTlc(T^*M)$ to $\cTlc(T^*N)$.

For an object $F \in \cT(T^*M)$, we define its reduced microsupport $\RS(F) \subset T^*M$ by 
\[
	\RS(F) \coloneqq \overline{\rho_t(\MS(F) \cap \{ \tau >0\}) },
\]
where $\rho_t \colon \{ \tau >0\} \to T^*M, (x,t;\xi,\tau)
\mapsto (x;\xi/\tau)$.  For a closed subset $A\subset T^*M$, we let $\cT_A(T^*M)$ be the full subcategory of $\cT(T^*M)$ consisting of the $F$ with $\RS(F) \subset A$.  We also set $\cT_{\mathrm{lc},A}(T^*M) = \cT_A(T^*M) \cap \cTlc(T^*M)$.

\section{Preliminaries}\label{sec:preliminaries}

We recall that we have a quantization map for Hamiltonian isotopies $Q \colon \DHam_c(T^*M) \to \cT(T^*M^2)$ introduced in~\cite{G-K-S}. It is defined so that $\RS(Q(\varphi))$ is the graph of $\varphi$. For $\varphi \in \DHam_c(T^*M)$ and $\cK_\varphi = Q(\varphi)$, the action of $\cK_\varphi$ on $\cT(T^*M)$, $F \mapsto \cK_\varphi^{\convstar}(F)= \cK_\varphi \convstar F$, is an auto-equivalence of category and we have $\RS(\cK_\varphi^{\convstar}(F)) = \varphi(\RS(F))$.  The category $\cTlc(T^*M^2)$ is not a group but it comes with the operation $\convstar$ which is associative and has $\bfk_{\Delta_M \times [0,+\infty[}$ as a unit element. Then $Q$ respects the operations on $\DHam_c(T^*M)$ and $\cTlc(T^*M^2)$: $Q(\varphi \circ \psi) \simeq Q(\varphi) \convstar Q(\psi)$.

We also have a quantization map for smooth compact exact Lagrangians, denoted by the same letter, $Q\colon \mathcal{L}(T^*M) \to \cT(T^*M)$ defined more recently in \cite{guillermou19, Viterbo-Sheaves}, constructed so that $\RS(Q(L))=L$ for any $L \in \cL(T^*M)$. This functor is an isometric embedding for the spectral and interleaving distances respectively (see \cite[prop 6.3]{GV2022singular}): for $L_1, L_2 \in \cL(T^*M)$,
\[
 d_{\cT(T^*M)}(Q(L_1), Q(L_2)) = \gamma(L_1,L_2) .
\]
Since the map $Q$ is an isometry, it extends to the completion\footnote{Note that the image is complete, but the map is not onto.} as an isometric embedding $\widehat Q \colon \widehat{\cL}(T^*M) \to \cT(T^*M)$ defined in~\cite{GV2022singular}.
We notice that $\widehat Q(T_c (\widetilde L_\infty)) \simeq {T_c}_* \widehat Q(\widetilde L_\infty)$.  The main result of~\cite{AGHIV} is the following connection between microsupport, $\gamma$-support and quantization:
\[
\RS(\widehat Q( \widetilde L_\infty)) 
= \gammasupp(\widetilde L_\infty) 
\quad \text{for any  $\widetilde L_\infty \in \widehat\cL(T^*M)$.}
\]

An approximation argument is missing in \cite{GV2022singular}, which we shall now provide. 

\begin{proposition}
	For any $\widetilde L_\infty \in \widehat{\cL}(T^*M)$, one has $\widehat{Q}(\widetilde L_\infty) \in \cTlc(T^*M)$.
\end{proposition}
\begin{proof}
	According to \cite{Cieliebak-Eliashberg}, Corollary 6.25, an element $\widetilde L \in {\cL}(T^*M)$, is $C^k$-approximated for any $k\geq 1$ by analytic Lagrangians $\widetilde L_i$. We thus find that $\widetilde L=C^k-\lim \widetilde L_i$ hence $\widetilde L_i=\varphi_i(\widetilde L)$ where $\varphi_i$ is generated by a $C^k$-small Hamiltonian. According to \cref{lem:Ham-approx} the distance between $\widetilde L$ and $\widetilde L_i$ can then be chosen arbitrarily small. As a result $\widetilde L$ is a $\gamma$-limit of analytic Lagrangians. 
    According to \cite{KS90} Theorem~8.4.2, the $Q(\widetilde L_i)$ are constructible hence their limit $Q(\widetilde L)$ is in $\cTlc(T^*M)$. 
    Since $\widetilde L_\infty$ can be written as a Cauchy sequence of elements of $\cL(T^*M)$, the claim follows.
\end{proof}

\begin{lemma}\label{lem:Ham-approx}
	Let $h \colon (T^*(M \times \bR) \setminus 0_{M\times \bR}) \times I \to \bR$ be a homogeneous Hamiltonian function and $\phi$ be the associated homogeneous Hamiltonian isotopy. 
	Let $K \in \SD((M \times \bR)^2 \times I)$ be the sheaf associated with $\phi$ constructed in~\cite{G-K-S}. 
	Then, for any $F \in \cT(T^*M)$
	\[
	d_{\cT(T^*M)}(F,K_1 \circ F) \le 4\int_0^1 \max |h_s(x,t;\xi,1)| \, ds.
	\]
\end{lemma}
\begin{proof}
    First note that we have
    \begin{align*}
        \MS(K) \subset 
        \big\{
            & (\phi_s(x,t;\xi,\tau),(x,t;-\xi,-\tau), (s,-h_s(\phi_s(x,t;\xi,\tau))))
            \; \big| \\
            & (x,t;\xi,\tau) \in T^*(M \times \bR) \setminus 0_{M\times \bR}, s \in I
        \big\}, 
    \end{align*}
    which implies 
    \[
        \MS(K \circ F) 
        \subset T^*M 
        \times 
        \left\{ (t,s;\tau,\sigma) \; \middle| \;
        \tau \ge 0,
        -\max h_s(x,t;\xi,\tau) \le \sigma \le -\min h_s(x,t;\xi,\tau)
        \right\}. 
    \]
    Since $h$ is homogeneous, we get $h_s(x,t;\xi,\tau) = \tau h_s(x,t;\xi/\tau,1)$ for $\tau>0$.
    Thus, we can apply the same proof in Theorem~4.16 \cite{AI20} to get 
    \[
        d_{\cT(T^*M)}(F,K_1 \circ F) \le 2 \int_0^1 \left( \max h_s(x,t;\xi,1) - \min h_s(x,t;\xi,1) \right) \, ds.
    \]
    Here, note that the distance $d_{\cT(T^*M)}$ is slightly different from the distance $d_{\cD(M)}$ in \cite{AI20} and we have $d_{\cT(T^*M)} \le 2d_{\cD(M)}$.
    The right-hand side of the inequality is bounded above by the desired integral. 
\end{proof}

Let $L \in \mathfrak L(T^*M)$ and $L_\infty \in \widehat\frakL(T^*M)$. We assume that
$\gammasupp(L_\infty)=L$ and we want to prove that $L_\infty\in \frakL(T^*M)$ and $L_\infty = L$.  Let
$\widetilde L = (L, f_L, \widetilde G) \in \cL(T^*M)$ be a lift of $L$ and let $\widetilde L_\infty \in
\widehat \cL(T^*M)$ be a lift of $L_\infty$.  In view of~\cite{AGHIV} our assumption means that the sheaf
$F_{\widetilde L_\infty} = \widehat Q(\widetilde L_\infty) \in \cT(T^*M)$ satisfies $\RS(F_{\widetilde
  L_\infty}) = L$.  To prove that $L_\infty = L$ it is enough to see that $F_{L_\infty}$ is isomorphic to $F_L
= Q(\widetilde L)$, up to translation (in $t$) and shift (in grading).  To this end, we shall characterize the
objects $F$ of $\cT_{\mathrm{lc},L}(T^*M)$ with $\MS(F) = T_c(\Lambda)$ for some $c \in \bR$, where $\Lambda \subset
T^*(M\times\bR)$ is the cone over a Legendrian lift of $L$ and $T_c$ also denotes the translation on $T^*(M \times \bR)$ by $c$ (see \cref{def:minimalSS}, \cref{lem:caractSSminimal0} and \cref{prop:caracterisation-F_L}).  
Explicitly
\[
    \Lambda= \{(x,\tau p, -f_{L}(x,p), \tau) \mid \tau >0, (x,p)\in L\}.
\] 
Hence $\Lambda$ is a conic Lagrangian submanifold of $T^*(M\times\bR)$ contained in $\{\tau > 0\}$.  Note that the
coisotropic submanifold $\rho_t^{-1}(L)$ is foliated by the translates of $\Lambda$: $\rho_t^{-1}(L) =
\bigsqcup_{c\in\bR} T_c(\Lambda)$.  
It is not too difficult to see that any closed conic coisotropic subset of $\rho_t^{-1}(L)$ is a union of
translates of $\Lambda$. 
Hence for any non zero $F\in \cT_L(T^*M)$, $\MS(F)$ contains at least $T_{c}(\Lambda)$ for some $c\in\bR$. However we shall not use these facts.

\section{Cohomologically chordless sheaves}

The main result we want to prove, \cref{thm:main}, is about the space $\widehat{\cL}(T^*M)$ and its statement is independent of sheaves. 
However, our proof starts by embedding this space in the category of sheaves via the functor $Q$. 
This embedding $Q$ is far from being essentially surjective.
We do not try to characterize its image, but we give here a useful property, \emph{cohomologically chordless}, shared by the sheaves in its image. 
This property is a cohomological consequence of the following geometric property: if $F = Q(\widetilde L)$ for some smooth Lagrangian brane $\widetilde L = (L,f_L, \widetilde G)$, then the reduction map $\MS(F) \cap ST^*(M\times\bR) \to T^*M$ is an embedding with image $L$. 
In other words, the Legendrian $\MS(F) \cap ST^*(M\times\bR)$ has no Reeb chords.  
Unfortunately, this geometric property is not necessarily preserved by taking limits since a $\gamma$-support may have double points (see Ex.~6.22 in~\cite{viterbo2022supports}).
However, the geometric property easily implies the following (already used in~\cite[chapter XII.4]{guillermou19}), which is stable by completion: $\RHom(F, T_{c*}F)$ is constant when $c$ runs over $\bR_{>0}$ or over $\bR_{<0}$ (and in the latter case it is zero).
Our \cref{def:minimalSS} below only retains the case $\bR_{<0}$ but gives a slightly stronger version.

As already mentioned, even for a cohomological chordless sheaf $F$, the map $\MS(F) \cap ST^*(M\times\bR) \to \RS(F)$ may not be injective. However we can give a sheafy statement analog to our main theorem: if $F$ is cohomological chordless and $\RS(F)$ is a smooth exact Lagrangian submanifold, then $\MS(F) \cap ST^*(M\times\bR) \to \RS(F)$ is a bijection (see \cref{prop:caracterisation-F_L} below for a more precise statement).

We denote by $q\colon M\times\bR \to M$ the projection. 

\begin{definition}\label{def:minimalSS}
  Let $F\in \cT(T^*M)$. We say that $F$ is \emph{cohomologically chordless} if
\begin{equation*}
  \RHom(F \otimes q^{-1}G, T_{c*}F ) \simeq 0
\end{equation*}
for all $c<0$ and all  locally constant  $G \in \SD(\bfk_{M})$ (we say that an object of $\SD(\bfk_{M})$ is locally constant if its cohomology sheaves are locally constant). 
\end{definition}

Before proving \cref{prop:caracterisation-F_L} we give several results about cohomologically chordless sheaves. 
\begin{lemma}
\label{lem:caractSSminimal0}
  Let $F\in \cT_L(T^*M)$ with $\MS(F) = T_{c_0}(\Lambda)$ for some $c_0\in \bR$ and $F|_{M\times \{t\}} \simeq 0$
  for $t\ll0$. Then $F$ is cohomologically chordless.
\end{lemma}
\begin{proof} 
  This is already done in~\cite[Lemma~12.4.4]{guillermou19}, but we sketch the proof for the convenience of the
  reader.  First the microsupports of $F \otimes q^{-1}G$ and $T_{c*}F$ do not meet when $c$ runs over
  $\mathopen]-\infty,0[$, hence $\RHom(F \otimes q^{-1}G, T_{c*}F)$ is independent of $c<0$ by a variation on
  the Morse theorem for sheaves~\cite[Corollary~5.4.19]{KS90} (see~\cite{nadler2016noncharacteristic} or~\cite[Corollary~1.2.17]{guillermou19}).  
  We choose $a$ such that $\Lambda \subset T^*(M \times
  \mathopen]-a,a[)$. For $c< -2a$ we obtain that $T_{c*}F$ is locally constant on $\supp(F \otimes q^{-1}G)$,
  say $T_{c*}F \simeq q^{-1}G' \simeq q^!G'[-1]$ there.  Then $\RHom(F \otimes q^{-1}G, T_{c*}F)$ is
  isomorphic to $\RHom(F \otimes q^{-1}G,q^!G'[-1])$. Using the adjunction $(\rR q_!,q^!)$ and the projection
  formula $\rR q_!(F \otimes q^{-1}G) \simeq \rR q_!F \otimes G$, it is then enough to check that $\rR q_!F
  \simeq 0$. This can be proved stalkwise: $(\rR q_!F)_x \simeq \RGamma_c(\{x\}\times\bR;
  F|_{\{x\}\times\bR})$ and the vanishing follows again from the Morse theorem for sheaves since
  $\MS(F|_{\{x\}\times\bR}) \subset \{\tau \geq 0\}$ and $F|_{\{x\}\times\bR}$ vanishes near $-\infty$.
\end{proof}

\begin{lemma}\label{lem:lim_minimalSS}
  Let $(F_i)_{i\in\bN}$, be a convergent sequence in $\cT(T^*M)$ and set $F = \lim_i F_i$ (the limit being for the distance $d_{\cT(T^*M)}$). 
  We assume that $F_i$ is cohomologically chordless for each $i\in \bN$.  
  Then $F$ is cohomologically chordless.
\end{lemma}
\begin{proof}
    By \cite[Proposition~6.25]{GV2022singular} (or~\cite[Theorem~4.3]{AI22completeness}), up to taking a subsequence, there
    exist a sequence of positive numbers $(\varepsilon_i)_{i\in \bN}$ converging to $0$ and morphisms 
    \begin{equation}\label{eq:sequence_morphisms}
        f_i \colon T_{-\varepsilon_i *}F_i \to T_{-\varepsilon_{i+1} *}F_{i+1}, 
        \quad 
        u_i \colon T_{-\varepsilon_i*}F_i \to F
    \end{equation}
    such that $u_{i+1} \circ f_i = u_i$, for all $n$, and the morphism $\hocolim
    T_{-\varepsilon_i *}F_i \to F$ induced by the $u_i$'s is an isomorphism, where $\hocolim$ is the sequential homotopy colimit described in \cite{bokstedt1993homotopy} (see also~\cite[Notation~10.5.10]{K-S06}).  The same proposition holds with
    homotopy limits instead of homotopy colimits and we can write in the same way (taking a subsequence again) $F \simto \holim T_{\eta_j *}F_j$ for some other sequence $(\eta_i)_{i\in \bN}$.
    
    Since the tensor product commutes with direct sums, it also commutes with homotopy colimits and we have, for
    any $G\in \SD(\bfk_M)$, $F \otimes q^{-1}G \simeq \hocolim ( T_{-\varepsilon_i *}F_i \otimes
    q^{-1}G)$. Recall that the category of sheaves on a topological space $X$ is a Grothendieck category, so we may apply \cref{lem:limcolim} and infer that $\RHom(F \otimes q^{-1}G, T_{c*}F)$ is a
    homotopy limit of $E_i = \RHom(T_{-\varepsilon_i *}F_i \otimes q^{-1}G, T_{(\eta_i+c)*}F)$.  For a given $c<0$ and for $i$
    big enough we have $\varepsilon_i + \eta_i +c <0$ and then $E_i \simeq 0$. It follows that $\RHom(F \otimes
    q^{-1}G, T_{c*}F)$ vanishes.
\end{proof}

\begin{lemma}\label{lem:limcolim}
  Let $\scC$ be a Grothendieck category.  Let $(A_i, f_i)$, $i\in \mathbb N$, be an inductive system in
  $\SD(\scC)$, with homotopy colimit $A$, and let $(B_j, g_j)$, $j\in \mathbb N$, be a projective system, with
  homotopy limit $B$.  Then $\RHom(A,B)$ is a homotopy limit of the system $(\RHom(A_i, B_i), h_i)$ where
  $h_i$ is the morphism induced by composition with $f_i$, $g_i$.
\end{lemma}
\begin{proof}
  According to \cite{Hovey}, Theorem~2.2, the category $\Ch (\scC)$ of chain complexes on $\scC$ is a model
  category having homotopical category $\SD(\scC)$. We denote by $\mathbb V_\bfk$ the category of
  $\bfk$-vector spaces.

  We apply results of \cite{Chacholski-Scherer} where homotopy (co)limits are defined for categories with weak
  equivalences.  If $\scA$ is such a category and $I$ is a small category, we have a functor $\holim'_I\colon
  \mathrm{Fun}(I,\scA) \to \mathrm{Ho}(\scA)$ (in particular $\holim'_I \colon \mathrm{Fun}(I,\Ch(\scC)) \to \mathrm{Ho}(\Ch(\scC)) = \SD(\scC)$.
  In the proof of \cref{lem:lim_minimalSS} the notation $\holim_I F$ applies to $F\in \mathrm{Fun}(I,\SD(\scC))$
  --- this is not a functor: $\holim_I F$ is well-defined up to a non-unique isomorphism.  We use the notation $\holim'_I F$ to avoid confusion (but this is denoted by $\holim_I$ in \cite{Chacholski-Scherer}).  We
  have $\holim'_I F \simeq \holim_I Q\circ F$ where $Q\colon \scA \to \mathrm{Ho}(\scA)$ is the quotient.

  We will apply Section~31.5 from \cite{Chacholski-Scherer} which states that if $F\colon I\times J \to \scC$ is a functor to a model category, then $\holim'_{I\times J}F \simeq \holim_I\holim'_J F
  \simeq\holim_J\holim'_I F$.  In our case $I=J = \mathbb N^\mathrm{op}$.  We first lift the diagram $i\mapsto A_i$ to a
  similar diagram in the set of chain complexes on $\scC$. We shall use the same notation for the lift. We do
  the same for $j \mapsto B_j$ and we may further impose that each $B_j$ is a complex of injectives. We then
  define a functor $F\colon (\mathbb N^\mathrm{op})^2 \to \Ch (\mathbb V_\bfk)$ by $F(i,j)=\Hom (A_i,B_j)$.
  Since the $B_j$'s are injective, we have $\holim'_i F(i,j) \simeq \RHom(\hocolim'_i A_i, B_j) \simeq \RHom(A,
  B_j)$ for each $j$.  From the definition of $\holim$ we also have $\holim_j \RHom(A, B_j) \simeq \RHom(A,
  \holim'_j B_j)$. Hence
  \[
    \RHom(A,B)\simeq\holim'_{(i,j)\in (\mathbb N^\mathrm{op})^2} \Hom(A_i,B_j).
  \]
  According to 31.6 (loc.\ cit.) for $F\colon I \to \scC$ a functor in a model category and $f\colon J \to I$ an initial functor, the map 
  \[
    \holim'_I F \to \holim'_J f^*F
  \]
  is a weak equivalence. 
  Using the fact that the inclusion of the diagonal $\mathbb N^\mathrm{op}$ in $(\mathbb N^\mathrm{op})^2$ is initial we get
  \[
    \RHom(A,B)\simeq \holim'_{i\in \mathbb N^\mathrm{op}} \Hom(A_i,B_i) \simeq \holim_{i\in \mathbb N^\mathrm{op}} \RHom(A_i,B_i).
  \]
  This concludes the proof. 
\end{proof}

We now prove that, if $F\in \cT_{\mathrm{lc},0_M}(T^*M)$ is cohomologically chordless, then $\MS(F) = 0_M\times
(\{c_0\}\times ]0,\infty[)$ for some $c_0\in \bR$ (\cref{prop:caracterisation-F_L-0}
below). 

We first recall a microlocal characterization of the inverse image of sheaves by a projection with
contractible fibers. 
\begin{lemma}\label{lem:inv_im_contract_fib}
  Let $N$ be a manifold and let $I$ be an open interval (or more generally a contractible manifold).  Let
  $p\colon N\times I \to N$ be the projection and let $i_a\colon N\times \{a\} \to N \times I$ be the
  inclusion, for $a\in I$.  Then $p^{-1}\colon \SD(\bfk_N) \to \SD_{T^*N \times 0_I}(\bfk_{N \times I})$ is an
  equivalence of categories, with inverses $\rR p_*$ and $i_a^{-1}$, $a\in I$. Moreover, in the case $N=\bR$,
  these functors induce equivalences $\cT(\pt) \simeq \cT_{0_I}(T^*I)$ and $\cTlc(\pt) \simeq
  \cT_{\mathrm{lc},0_I}(T^*I)$.
\end{lemma}
\begin{proof}
  Proposition~2.7.8 of~\cite{KS90} says that $p^{-1}$ and $\rR p_*$ give equivalences between $\SD(\bfk_N)$
  and $\SD_p(\bfk_{N \times I})$, where the latter category is the subcategory of $\SD(\bfk_{N \times I})$
  whose objects restrict to constant sheaves on the fibers.  Now Proposition~5.4.5 of~\cite{KS90} says that
  $\SD_p(\bfk_{N \times I})$ coincides with $\SD_{T^*N \times 0_I}(\bfk_{N \times I})$.  Since $i_a^{-1} \circ
  p^{-1} \simeq \id_{\SD(\bfk_N)}$, we deduce that $i_a^{-1}$ is also an inverse to $p^{-1}$.

  The functors $p^{-1}$ and $i_a^{-1}$ commute with $- \cstar \bfk_{[0,+\infty[}$ and we deduce $\cT(\pt)
  \simeq \cT_{0_I}(T^*I)$.  Moreover they send constructible sheaves to constructible sheaves and are
  $1$-Lipschitz with respect to the interleaving distance. Hence they also induce the last equivalence of the lemma.   
\end{proof}

\begin{proposition}\label{prop:caracterisation-F_L-0}
  Let $F \in \cT_{\mathrm{lc},0_M}(T^*M)$ such that $F$ is cohomologically chordless.  
  Then there exists $c_0$ and a locally constant sheaf $G_0$ on $M$ such that $F \simeq
  G_0\boxtimes \bfk_{[c_0,+\infty[}$.
\end{proposition}
\begin{proof}
  (i) For $c\in\bR$ we set $G'_c = \rR q_* \cRHom(F, T_{c*}F)$.  By \cref{lem:inv_im_contract_fib}, for any
  open ball $B\subset M$, we have $F|_{B \times \bR} \simeq p^{-1}F'$ for some $F' \in \cTlc(\pt)$, where
  $p\colon B\times\bR \to \bR$ is the projection. It follows that $\cRHom(F, T_{c*}F)|_{B \times \bR} \simeq p^{-1}\cRHom(F', T_{c*}F')$. By base change we deduce that $G'_c|_{B}$ is constant. Hence
  $G'_c$ is locally constant.  We also have the adjunction isomorphism
    \begin{align*}
      \RHom(F \otimes q^{-1}G'_c, T_{c*}F)
      &\simeq  \RHom( q^{-1}G'_c, \cRHom(F, T_{c*}F)) \\
      &\simeq \RHom(G'_c, \rR q_*  \cRHom(F, T_{c*}F)) = \RHom(G'_c,G'_c).    
    \end{align*}
    Since $F$ is cohomologically chordless, it follows that $G'_c \simeq 0$ for any $c<0$.

    \smallskip
    \noindent (ii) Let $x \in M$ be given and let $B$ be a small ball around $x$.  With the same notations as in~(i) we have $\cRHom(F, T_{c*}F)|_{B \times \bR} \simeq p^{-1}\cRHom(F', T_{c*}F')$ and the base change formula gives $\RGamma(B; G'_c) \simeq \RHom(F', T_{c*}F')$.
    For $c<0$ we thus obtain $\RHom(F',T_{c*}F') \simeq 0$.  Let us check that this implies $F' \simeq
    E\otimes \bfk_{[c_0,+\infty[}$ for some constant sheaf $E$ on $\mathbb R$ and some $c_0\in \bR$.

    By~\cite[Corollary~B.12]{GV2022singular} we have a decomposition $F' \simeq \bigoplus_{j \in \scI} \bfk_{[a_j,b_j[}[d_j]$,
    where $\scI$ is a countable set and $a_j\in \bR$, $b_j \in \bR \cup \{+\infty\}$, $d_j \in \bZ$.  If $F'$
    is not of the form $E\otimes \bfk_{[c_0,+\infty[}$, then there exists $n$ with $b_n\not=+\infty$ or there exist $n,m$
    with $a_n\not=a_m$ (say $a_n<a_m$).  In the first case we write $F' \simeq \bfk_{[a_n,b_n[}[d_n] \oplus
    F''$ and see that $H(c) := \RHom( \bfk_{[a_n,b_n[}, \bfk_{[c+a_n,c+b_n[})$ is a direct summand of
    $\RHom(F', T_{c*}F')$.  By \cref{lem:intervalles} below $H(c) \simeq \bfk[-1]$ for $a_n-b_n < c <0$.
    The second case is similar, with the use of the fact that $\Hom(\bfk_{[a_n,+\infty[},
    \bfk_{[c+a_m,+\infty[}) \simeq \bfk$ for $a_n-a_m < c$.  In both cases we have $\RHom(F', T_{c*}F') \not=
    0$ and get a contradiction. Hence $F' \simeq E\otimes \bfk_{[c_0,+\infty[}$ for some constant sheaf $E$ and $c_0\in
    \bR$, as claimed.

    \smallskip\noindent (iii) Summing up, we have for any $x\in M$ and ball $B$ around $x$ an isomorphism
    $F|_{B \times \bR} \simeq p^{-1}(E\otimes \bfk_{[c_0,+\infty[})$ where $p\colon B\times\bR \to \bR$ is the projection,
    $E \in \SD(\bfk)$ and $c_0\in \bR$.  Since $M$ is connected, $c_0$ does not depend on $x$.  It follows
    that $F$ is supported on $M \times [c_0,+\infty[$, hence $\cRHom(\bfk_{M \times [c_0,+\infty[}, F) \simto F$.
    
    Let us set $G_0 = \rR q_*F$.  The image of $\id_{G_0}$ by the adjunction isomorphisms
    \begin{multline*}
      \Hom(G_0, \rR q_*F) \simeq   \Hom(q^{-1}G_0, F) \\
      \simeq \Hom(q^{-1}G_0, \cRHom(\bfk_{M \times [c_0,+\infty[}, F)) 
      \simeq \Hom(q^{-1}G_0 \otimes \bfk_{M \times [c_0,+\infty[}, F)
    \end{multline*}
    gives a morphism $u\colon q^{-1}G_0 \otimes \bfk_{M \times [c_0,+\infty[} \to F$. By~(ii) it is locally an
    isomorphism, hence it is an isomorphism.
\end{proof}

\begin{lemma}\label{lem:intervalles}
  Let $a,c \in \bR$ and $b,d \in \bR \cup \{+\infty\}$ with $a<b$, $c<d$. We have
  \begin{align*}
  \cRHom(\bfk_{[a,b[}, \bfk_{[c,d[}) &\simeq \cRHom(\bfk_{[a,b[ \, \cap \,]c,d]}, \bfk_{\bR}) \\[1mm]
  &\simeq
  \begin{cases}
    \bfk_{[c,b]} & \text{if } a\leq c < b \leq d, \\
    \bfk_{]a,d[} & \text{if }  c < a < d < b , \\
    \bfk_{\{a\}}[-1] & \text{if } a = d, \\
    \bfk_I & \text{else, where $I$ is half closed or empty}
  \end{cases}
  \end{align*}
  and  in particular
  \[
  \RHom(\bfk_{[a,b[}, \bfk_{[c,d[}) \simeq
  \begin{cases}
  \bfk & \text{if } a\leq c < b \leq d, \\
  \bfk[-1] & \text{if }  c < a \leq d < b , \\
  0 & \text{else.}
  \end{cases}
  \]  
\end{lemma}
\begin{proof}
  For an interval $I$ with non empty interior let us write $I^* = (\overline{I} \setminus I) \cup \Int(I)$ (in
  words, we turn closed ends into open ones and conversely).  Then $\cRHom(\bfk_I,\bfk_\bR) \simeq \bfk_{I^*}$.
  In particular $\cRHom(\bfk_{[a,b[}, \bfk_{[c,d[}) \simeq \cRHom(\bfk_{[a,b[}, \cRHom(\bfk_{]c,d]}
  ,\bfk_\bR)) \simeq \cRHom(\bfk_{[a,b[} \otimes \bfk_{]c,d]} ,\bfk_\bR)$, which gives the first isomorphism.
  The second one follows by a case by case check, together with the additional isomorphism
  $\cRHom(\bfk_{\{a\}}, \bfk_\bR) \simeq \bfk_{\{a\}}[-1]$.  The last assertion is obtained by taking global
  sections.
\end{proof}

We now check that $\DHam_c(T^*M)$ and its completion preserve cohomologically chordless sheaves.
\begin{lemma}\label{lem:cohom_chord_Ham_invar}
	Let $\varphi \in \DHam_c(T^*M)$ and $\cK_\varphi = Q(\varphi)$.  
	Let $F\in \cT(T^*M)$ be cohomologically chordless. 
	Then $\cK_\varphi^\convstar(F)$ is cohomologically chordless.
\end{lemma}
\begin{proof}
	Since $\cK_\varphi^\convstar$ is an equivalence, we have 
	\[
	\RHom(\cK_\varphi^\convstar(F \otimes q^{-1}G), \cK_\varphi^\convstar(T_{c*}F)) 
	\simeq \RHom(F \otimes q^{-1}G, T_{c*}F).
	\]
	Hence it is enough to check that $\cK_\varphi^\convstar$ commutes with $T_{c*}$, which is clear by the definition of $\convstar$, and that
	\[
	\cK_\varphi^\convstar(F \otimes q^{-1}G) \simeq \cK_\varphi^\convstar(F) \otimes q^{-1}G.
	\]
	Since $\varphi$ is the time $1$ of some isotopy, both sides of this isomorphism are restrictions at time $1$ of sheaves in $\cT_A(T^*(M\times\bR))$, where $A \subset T^*(M\times\bR)$ is given by $A=\{(x,\xi,s,\sigma) \mid \sigma = h(x,\xi,s)\}$, with $h$ the Hamiltonian function of $\varphi$.  
    Both sheaves coincide at time $0$ and the result follows from a uniqueness property in this situation (see for example Corollary~2.1.5 in~\cite{guillermou19}).
\end{proof}

We equip $\DHam_c(T^*M)$ with the sheaf-theoretic spectral metric $\gamma^s$ defined as 
\[
	\gamma^s(\varphi,\varphi') = d_{\cT(T^*M^2)}(\cK_{\varphi},\cK_{\varphi'}).
\] 
Denote by $\widehat{\DHam_c}(T^*M)$ the completion of $\DHam_c(T^*M)$ with respect to $\gamma^s$. 
By the completeness of $\cT(T^*M^2)$ with respect to $d_{\cT(T^*M^2)}$ \cite{AI22completeness,GV2022singular}, we can extend the map $Q$ as 
\[
	\widehat{Q} \colon \widehat{\DHam_c}(T^*M) \to \cTlc(T^*M^2).
\]

As a completion of a group, $\widehat {\DHam_c}(T^*M)$ is a group and the formula $Q(\varphi \circ \psi) \simeq Q(\varphi) \convstar Q(\psi)$ given at the beginning of \S\ref{sec:preliminaries} extends to $\widehat{Q}$.

\begin{lemma}\label{lem:conv=equiv}
	Let $\varphi_\infty \in \widehat{\DHam_c}(T^*M)$ and $\cK_{\varphi_\infty} = \widehat Q(\varphi_\infty)$.  
	Then $\cK_{\varphi_\infty}^\convstar \colon \cT(T^*M) \to \cT(T^*M)$, $F \mapsto \cK_{\varphi_\infty}^\convstar(F) \coloneqq \cK_{\varphi_\infty} \convstar F$ is an equivalence of categories.  
	Moreover, if $F \in \cT(T^*M)$ is cohomologically chordless, so is $\cK_{\varphi_\infty}^\convstar(F)$.
\end{lemma}
\begin{proof}
	We find that $\cK_{\varphi_\infty}$ has an inverse with respect to $\convstar$ given by $\cK_{\varphi_\infty}^{-1} = \widehat Q(\varphi_\infty^{-1})$. 
	The first assertion is then clear.  
	Writing $\varphi_\infty$ as a limit of a Cauchy sequence $(\varphi_n)_n$ of $\DHam_c(T^*M)$, the sequence $\cK_{\varphi_n}^\convstar(F)$ converges to $\cK_{\varphi_\infty}^\convstar(F)$.
	Hence the second assertion follows from \cref{lem:cohom_chord_Ham_invar,lem:lim_minimalSS}.
\end{proof}

Now we extend \cref{prop:caracterisation-F_L-0} to the case of a general exact Lagrangian $L$. 
To reduce the problem to \cref{prop:caracterisation-F_L-0} we shall use  a result Arnaud, Humili\`ere, and Viterbo~\cite{AHV2024higher}. 

\begin{theorem}[{cf.~\cite{AHV2024higher}}]\label{thm:nearbyconjC0} 
    Let $L \in \frakL(T^*M)$ be a compact exact Lagrangian submanifold of $T^*M$.
    Then, there exists $\varphi_\infty \in \widehat{\DHam_c}(T^*M)$ such that $\varphi_\infty(L)=0_M$, where both sides should be understood as elements in $\widehat{\frakL}(T^*M)$.
	Moreover, the functor $\cK_{\varphi_\infty}^\convstar \colon \cT(T^*M^2) \to \cT(T^*M^2)$ sends $\cT_{L}(T^*M)$ to $\cT_{0_M}(T^*M)$.
\end{theorem} 

The theorem is proved in \cite{AHV2024higher} for the completion of $\DHam_c(T^*M)$ with respect to the usual spectral metric $\gamma$.
In \Cref{sec:weak_nearby_sheafgamma}, we give a proof for the sheaf-theoretic spectral metric~$\gamma^s$.
Note that if $\bfk$ is of characteristic $2$ or $M$ is spin, then $\gamma^s$ coincides with the usual spectral metric $\gamma$ (see \cite{GV2022singular}).

\begin{proposition}\label{prop:caracterisation-F_L}
  Let $\widetilde L \in \cL(T^*M)$ be a lift of $L \in \frakL(T^*M)$ and set $F_{\widetilde L} = Q(\widetilde L)$.  
  Let $\varphi_\infty \in \widehat {\DHam_c}(T^*M)$ be given by \cref{thm:nearbyconjC0} and set $\cK_{\varphi_\infty} = \widehat Q(\varphi_\infty) \in \cTlc(T^*M^2)$.
  \begin{enumerate}
      \item There exist a locally constant sheaf $G_0$ on $M$, of rank $1$, and $c_0\in \bR$ such that $\cK_{\varphi_\infty}^\convstar(F_{\widetilde L}) \simeq G_0\boxtimes \bfk_{[c_0,+\infty[}$. 
      \item\label{cor:FLii} Let $F \in \cT_{\mathrm{lc},L}(T^*M)$ such that $F$ is cohomologically chordless.  
      Then there exists $c_1$ and a locally constant sheaf $G_1$ on $M$ such that $F \simeq T_{c_1*}(F_{\widetilde L} \otimes q^{-1}G_1)$.
  \end{enumerate}
\end{proposition}
\begin{proof}
  Let $F$ be as in~(ii).  
  By \cref{thm:nearbyconjC0} $\cK_{\varphi_\infty}^\convstar(F) \in \cT_{\mathrm{lc},0_M}(T^*M)$, and by \cref{lem:conv=equiv} it is cohomologically chordless.  
  By \cref{prop:caracterisation-F_L-0}, we deduce $\cK_{\varphi_\infty}^\convstar(F) \simeq G_F\boxtimes \bfk_{[c_F,+\infty[}$ for some $c_F\in \bR$ and some locally constant sheaf $G_F$ on $M$.
  
  By \cref{lem:caractSSminimal0} the sheaf $F_{\widetilde L}$ satisfies the property in~(ii).  
  In particular, we have~(i) with $c_0 = c_{F_{\widetilde L}}$ and $G_0 = G_{F_{\widetilde L}}$, but we have to check that $G_0$ is of rank $1$.  
  We set $F_1 = (\cK_{\varphi_\infty}^{-1})^\convstar(\bfk_{[c_0,+\infty[})$.  Then
  $\cK_{\varphi_\infty}^\convstar(F_1 \otimes q^{-1}G_0) \simeq \cK_{\varphi_\infty}^\convstar(F_{\widetilde L})$. Hence $F_1 \otimes q^{-1}G_0
  \simeq F_{\widetilde L}$.  Restricting to $M\times \{t\}$ for $t\gg0$ we obtain $(F_1|_{M\times \{t\}})
  \otimes G_0 \simeq F_{\widetilde L}|_{M\times \{t\}}\simeq \bfk_M$, which implies that $G_0$ is of rank
  $1$.  This proves~(i).

  Now we come back to $F$ as in~(ii). We have
  \[
    \cK_{\varphi_\infty}^\convstar(F) \simeq  G_F\boxtimes \bfk_{[c_F,+\infty[}
    \simeq T_{c_1*}(G_1 \otimes ( G_0\boxtimes \bfk_{[c_0,+\infty[})) ,
  \]
  where $c_1 = c_F-c_0$ and $G_1 = G_F \otimes G_0^{-1}$. 
  Hence $\cK_{\varphi_\infty}^\convstar(F) \simeq \cK_{\varphi_\infty}^\convstar(T_{c_1*}(F_{\widetilde L} \otimes q^{-1}G_1))$ and~(ii) follows.
\end{proof}

\section{Regular Lagrangians are smooth Lagrangians}

We now prove our first main result, as a corollary of the characterization of cohomologically chordless sheaves obtained in \cref{prop:caracterisation-F_L}.

\begin{theorem}\label{thm:main}
  Let $L_\infty \in \widehat\frakL(T^*M)$. We assume that $L = \gammasupp(L_\infty)$ is a compact exact Lagrangian
  submanifold of $T^*M$.  Then $L_\infty = L$ in $\frakL(T^*M)$.
\end{theorem}
\begin{proof}
  We lift $L_\infty$ to $\widetilde L_\infty \in \widehat{\cL}(T^*M)$ and set $F_{\infty} = \widehat Q
  (\widetilde L_\infty)$.  
  By definition, $\widetilde L_\infty$ is the equivalence class of a Cauchy sequence $(L_n)_{n}$ in $\cL(T^*M)$ and, by~\cite{GV2022singular} or~\cite{AI22completeness}, the sequence of associated
  sheaves $F_{L_n}$ converges in $\cT(T^*M)$ to $F_{\infty}$.  By~\cite{AGHIV} we know that $\RS(F_{\infty}) =
  L$.  We lift $L$ into $\widetilde L \in\cL(T^*M)$ and let $F_{\widetilde L}$ be the associated sheaf.  By
  \cref{lem:lim_minimalSS} and \cref{prop:caracterisation-F_L} we have $F_{\infty} \simeq
  T_{c_1*}(F_{\widetilde L} \otimes q^{-1}G_1)$ for some $c_1\in\bR$ and some locally constant sheaf $G_1$ on $M$.  
  Let $x\in M$ be given.  The sequence $F_{L_n}|_{\{x\}\times\bR}$ converges to
  $F_{\infty}|_{\{x\}\times\bR}$ and $\RGamma(\{x\}\times\bR; F_{L_n}|_{\{x\}\times\bR}) \simeq \bfk$ for all $n$.  
  Hence $\RGamma(\{x\}\times\bR; F_{\infty}|_{\{x\}\times\bR}) \simeq \bfk$. It follows that $(G_1)_x \simeq \bfk$.  The same kind of argument shows that $\RGamma(M; G_1) \simeq \RGamma(M\times\bR; F_{\infty})
  \simeq \bfk$. Hence $G_1 \simeq \bfk_M$ and $F_{\infty} \simeq T_{c_1*}F_{\widetilde L}$, which implies
  $\widetilde L_\infty = T_{c_1}(\widetilde L)$, hence $L_\infty = L$.
\end{proof}

\begin{remark}
    Note that for $L=0_M$ in $T^*M$ and for a manifold $M$ satisfying a certain condition (denoted by $(\star)$ in \cite{Viterbo-inverse-reduction}), this theorem follows from Theorem~8.6 of \cite{viterbo2022supports}(version~2) as a consequence of Theorem~6.3 in \cite{Viterbo-inverse-reduction}. 
    This was removed from the published version of \cite{viterbo2022supports} and included in \cite{Viterbo-inverse-reduction}. 
\end{remark}

\section{Compact \texorpdfstring{$\gamma$}{gamma}-supports are connected}
Our second main result answers a question in \cite{AHV2024higher}. 
\begin{theorem}\label{thm:gammasupp_connected}
	Let $L_\infty \in \widehat{\frakL}(T^*M)$ and assume that $\gammasupp(L_\infty)$ is compact. 
	Then $\gammasupp(L_\infty)$ is connected.  
\end{theorem}

\begin{lemma}\label{lem:Linfty_end}
	For any $\widetilde L_\infty \in \widehat{\cL}(T^*M)$, one has $\End_{\cT(T^*M)}(F_{\widetilde L_\infty}) \simeq \bfk$.
\end{lemma}
\begin{proof}
    Write $\widetilde L_\infty$ as the equivalence class of a Cauchy sequence $(\widetilde L_n)_n$, where $\widetilde L_n \in \cL(T^*M)$. 
	Then, we find that $\rR q_* \homst(F_{\widetilde L_n},F_{\widetilde L_n}) \simeq \RGamma(M;\bfk_M) \otimes \bfk_{[0,+\infty[}$ for any $n$ by \cite[Proposition~9.11]{Viterbo-Sheaves}. 
	Since $\gamma(\widetilde L_n,\widetilde L_\infty) \to 0$ and $\homst$ is continuous for the interleaving distance, we find that $\rR q_*\homst(F_{\widetilde L_n},F_{\widetilde L_n})$ converges to $\rR q_*\homst(F_{\widetilde L_\infty},F_{\widetilde L_\infty})$. 
	This implies that 
    \[
        d_{\cT(\pt)}(\rR q_* \homst(F_{\widetilde L_\infty},F_{\widetilde L_\infty}),\RGamma(M;\bfk_M) \otimes \bfk_{[0,+\infty[})=0 ,
    \]
    from which we deduce $\rR q_* \homst(F_{\widetilde L_\infty},F_{\widetilde L_\infty}) \simeq \RGamma(M;\bfk_M) \otimes \bfk_{[0,+\infty[}$ by \cite[Proposition~B.8]{GV2022singular}. 
	Thus, we obtain
	\[
		\Hom(F_{\widetilde L_\infty},F_{\widetilde L_\infty}) \simeq H^0\RHom(\bfk_{[0,+\infty[},\rR q_* \homst(F_{\widetilde L_\infty},F_{\widetilde L_\infty}))\simeq \bfk,
	\]
	which proves the lemma.
\end{proof}

Our next Lemma is a variant of microlocal cut-off lemma. A cut-off functor associated with an open subset $\Omega$ of a cotangent bundle sends a sheaf $F$ to a sheaf $F'$ such that $\MS(F') \subset \overline{\Omega}$ and $\MS(F') \cap \Omega = \MS(F)$. Such functors were first introduced in \cite{KS90} for special cases of $\Omega$ and more recently in \cite{Agnolo-cut-off, Chiu17,zhang2024capacities,zhang2023idempotence,Kuowrap, KSZ23,Bingyucut-off}. In general $\MS(F') \cap \partial\Omega$ will not be bounded by $\MS(F) \cap \partial\Omega$. However, when $\MS(F) \cap \partial\Omega$ is empty, we check that it holds true.

\begin{lemma}\label{lem:projector_cutoff}
    Let $U$ be an open subset of $T^*M$ and $F \in \cT(T^*M)$.
    Assume that $\RS(F)\cap U$ is compact and $\RS(F) \cap \partial U = \varnothing$, where  $\partial U$ is the topological boundary defined as $\overline{U} \setminus U$.
    Then one has an exact triangle 
    \[
        P(U) \convstar F \to F \to Q(U) \convstar F \xrightarrow{+1}
    \]
    with $\RS(P(U) \convstar F) = \RS(F) \cap U$ and $\RS(Q(U) \convstar F) = \RS(F) \cap (T^*M \setminus \overline{U})$.
    Here $P(U), Q(U) \colon \cT(T^*M) \to \cT(T^*M)$ are the microlocal projectors associated with $U$ (see \cite{Chiu17,zhang2024capacities,zhang2023idempotence}).
\end{lemma}
\begin{proof}
    Set $A_1 \coloneqq \RS(F) \cap U$ and $A_2 \coloneqq \RS(F) \setminus A_1$.
    To show $P(U) \convstar F \in \cT_{A_1}(T^*M)$ and $Q(U) \convstar F \in \cT_{A_2}(T^*M)$, we use (a version of) Kuo's description of projectors by microlocal wrapping~\cite{Kuowrap, KSZ23}. 
    
    Let $C_c^{\infty}(U)$ be the poset of compactly supported smooth functions on $U$ and $H_\bullet\colon \bN\to C_c^{\infty}(U)$ be a final functor satisfying $H_n\equiv n $ on a neighborhood of $A_1$ for each $n$. 
    The microlocal projector $Q(U)$ is described as  (see after~\eqref{eq:sequence_morphisms} for the notation $\hocolim$) 
    \[
        Q(U)\simeq \hocolim_{n\in \bN} \cK_{H_n} .
    \]

    Since $\hocolim$ commutes with $\convstar$, 
    $Q(U) \convstar F\simeq \hocolim_n (\cK_{H_n}\convstar F)$. 
    Since $d H_n$ vanishes on a neighborhood of $\RS (F)$, $\RS(\cK_{H_n}\convstar F) =\RS(F)=A_1\cup A_2$. 
    Hence $\RS(\hocolim_n (\cK_{H_n}\convstar F))\subset A_1\cup A_2$. 
    On the other hand, $\RS(Q(U) \convstar F)\subset T^*M\setminus U$ by a formal property of the projector $Q(U)$ and hence, 
    \[
        \RS(Q(U) \convstar F)\subset (T^*M\setminus U)\cap (A_1\cup A_2)=A_2.
    \]
    Since the morphism $F\to Q(U) \convstar F$ is an isomorphism on $T^*M\setminus \overline{U}$, 
    $(P(U) \convstar F) \cap (T^*M\setminus \overline{U})=\varnothing$. 
    The triangle inequality for microsupports shows $\RS (P(U) \convstar F)=A_1$.
\end{proof}

The next lemma is a wide generalization of \cite[Proposition~3.3.2]{guillermou19}, where the result is local and the sets $A_1$, $A_2$ are supposed ``unknotted''.

\begin{lemma}\label{lem:directsum_decomp}
	Let $F \in \cT(T^*M)$ and assume that $\RS(F)$ is decomposed into two compact disjoint subsets $A_1$ and $A_2$.
	Then there exist $F_1, F_2 \in \cT(T^*M)$ such that $\RS(F_i)=A_i$ and $F \simeq F_1 \oplus F_2$.
\end{lemma}	
\begin{proof}
	Take an open neighborhood $U$ of $A_1$ such that $\overline{U} \cap A_2 = \varnothing$.
    Applying \cref{lem:projector_cutoff}, we have an exact triangle in $\cT(T^*M)$
	\[
		P(U) \convstar F \to F \to Q(U) \convstar F \xrightarrow{+1}
	\]
	with $P(U) \convstar F \in \cT_{A_1}(T^*M)$ and $Q(U) \convstar F \in \cT_{A_2}(T^*M)$. 
	Set $F_1 \coloneqq P(U) \convstar F$ and $F_2 \coloneqq Q(U) \convstar F$. 
	Since $A_1 \cap A_2 = \varnothing$ and each $A_i$ is compact, by Tamarkin's separation theorem we have $\Hom_{\cT(T^*M)}(F_2,F_1[1])=0$. 
	Then by the above exact triangle, $F \simeq F_1 \oplus F_2$.
\end{proof}

\begin{proof}[Proof of \cref{thm:gammasupp_connected}]
	Suppose that $\gammasupp(L_\infty)$ is decomposed into two non-empty compact disjoint subsets $A_1$ and $A_2$.
    Let $\widetilde L_\infty$ be a lift of $L_\infty$.
	Set $F_{\widetilde L_\infty}=Q(\widetilde L_\infty) \in \cT(T^*M)$.  
	By a result of \cite{AGHIV}, we have $\gammasupp(L_\infty)=\RS(F_{\widetilde L_\infty})$. 
	By \cref{lem:directsum_decomp}, there exist $F_i \in \cT_{A_i}(T^*M)$ such that $F_{\widetilde L_\infty} \simeq F_1 \oplus F_2$.
	Since $F_{\widetilde L_\infty}$ is indecomposable by \cref{lem:Linfty_end}, either $F_1$ or $F_2$ is zero.
	This is a contradiction.
\end{proof}

\begin{remarks}
    \begin{enumerate}
        \item The connectedness does not hold for  elements in $\widehat{\mathfrak L}(T^*M)$ with non-compact $\gamma$-support. Indeed, consider the situation in \cref{Figure-cusp}.
        Since $f_j$ $C^0$-converges to $f$, $L_j=\mathrm{graph}(df_j)$ $\gamma$-converges to some $L_\infty$ in $\widehat{\mathfrak L}(T^*S^1)$. Clearly the $\gamma$-support of $L_\infty$ is the union of the two connected components represented in \cref{Figure-cusp}(b).
        \item One can construct examples of sequences $F_i$ of elements in $\cT(T^*M)$ such that $\RS(F_i)$ remain in a fixed compact set, are connected, $F_i$ converges to $F$ for $d_{\cT(T^*M)}$, but $\RS(F)$ is not connected. As a result $F$ is not in the image of $\widehat{Q}$.  
    \end{enumerate}
\end{remarks}

\begin{figure}[H]
	\centering
	\begin{minipage}[c]{0.49\hsize}
		\centering
		\tikz[overlay]{
			\node at (3.2,0.9) {$f_j$};
			\node at (4.5,1.5) {$f=\lim f_j$};
		}
		\includegraphics[width=7cm]{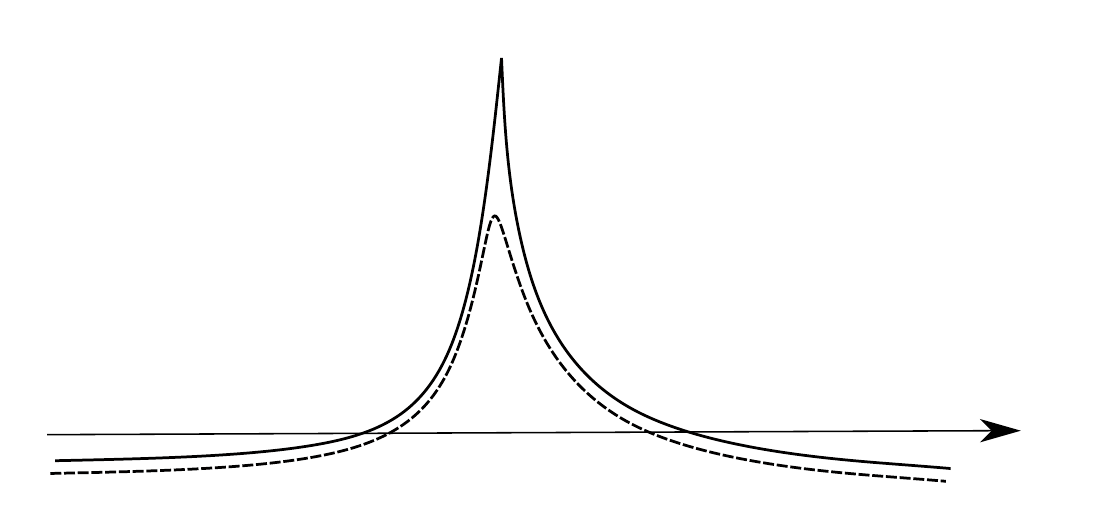}
		\subcaption{}
	\end{minipage}
	\begin{minipage}[c]{0.49\hsize}
		\centering
		\tikz[overlay]{
			\node at (2.2,2.9) {$L_j=\mathrm{graph} (df_j)$};
			\node at (2.9,4.3) {$L_\infty$};
		}
		\includegraphics[width=7cm]{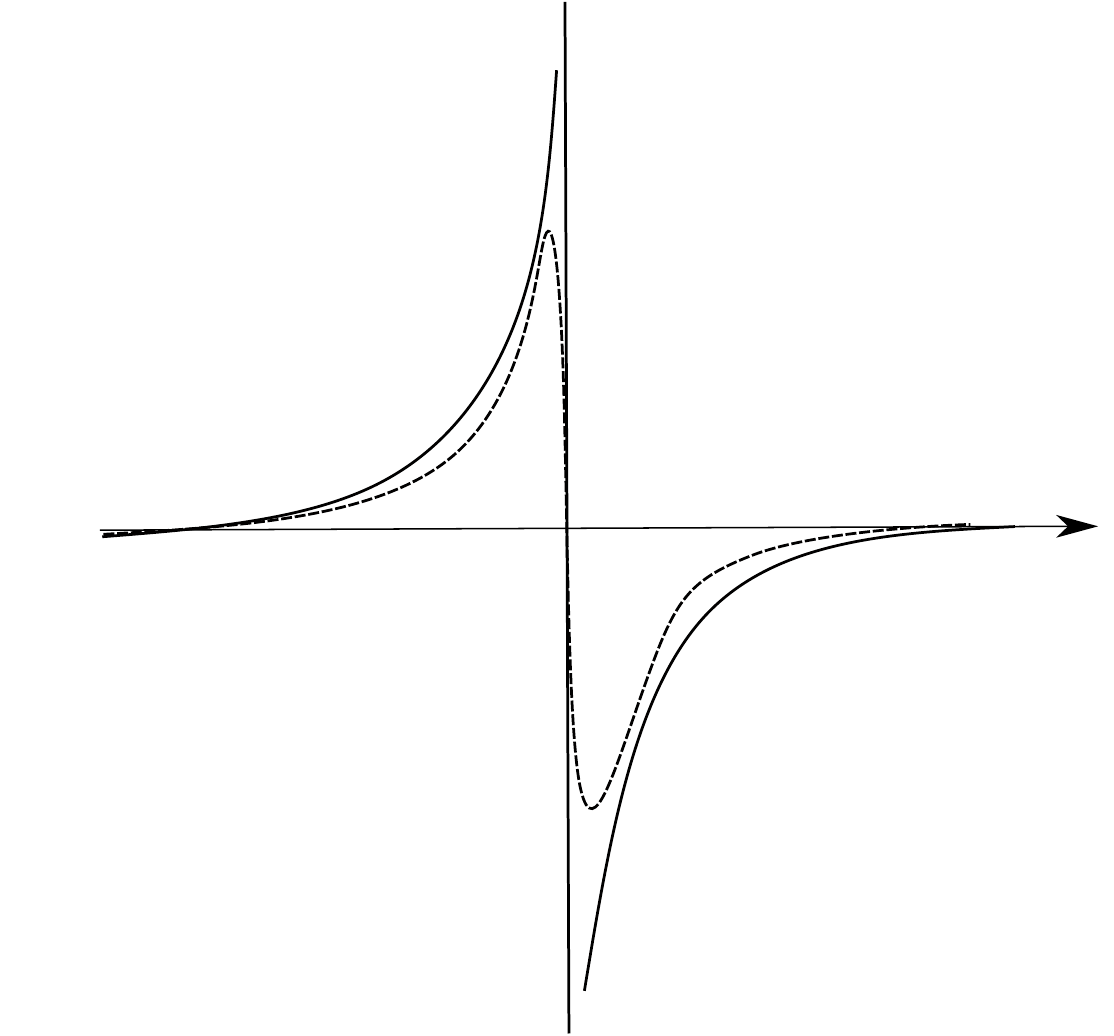}
		\subcaption{}
	\end{minipage}
	\caption{An element $L_\infty=\gamma\mathchar`-\lim(L_j)$ in $\widehat{\mathfrak L}(T^*S^1)$ with non-compact and disconnected $\gamma$-support. Figure (b) is the differential of Figure (a), the $f_j, L_j$ correspond to the dashed curves, $f,L_\infty$ to the solid curves. }\label{Figure-cusp}
\end{figure}

\appendix

\section{The weak nearby Lagrangian conjecture for the sheaf-theoretic spectral metric}\label{sec:weak_nearby_sheafgamma}

In this appendix, we prove a sheaf-theoretic version of a result of Arnaud, Humili\`ere, and Viterbo~\cite{AHV2024higher}. 
The following theorem is for the sheaf-theoretic spectral metric~$\gamma^s$.

\begin{theorem}[{cf.~\cite{AHV2024higher}}]\label{thm:nearbyconj_appendix} 
    Let $L \in \frakL(T^*M)$ be a compact exact Lagrangian submanifold of $T^*M$.
    Then, there exists $\varphi_\infty \in \widehat{\DHam_c}(T^*M)$ such that $\varphi_\infty(L)=0_M$, where both sides should be understood as elements in $\widehat{\frakL}(T^*M)$.
	Moreover, the functor $\cK_{\varphi_\infty}^\convstar \colon \cT(T^*M^2) \to \cT(T^*M^2)$ sends $\cT_{L}(T^*M)$ to $\cT_{0_M}(T^*M)$.
\end{theorem}

	The zero-section $0_M$ is the fixed point set of the canonical Liouville flow. 
	We set $\psi_0^s$ to be the time-$s$ map of the canonical Liouville flow and set $\psi_0\coloneqq \psi_0^1$. 
	The map $\psi_0^s$ is the multiplication by $e^s$ on each fiber. 
	By \cite[Proposition~7.3]{AHV2024higher}, we see that $L$ is also the fixed point set of the Liouville flow of another Liouville 1-form which coincides with the canonical Liouville form outside a compact subset. 
	Set $\psi_1$ to be the time-1 map of this latter Liouville flow.
	
	We will find $\varphi_\infty$ as a fixed point of a contraction on $\widehat{\DHam_c}(T^*M)$. 
	To construct the contraction, we first note the following.    
	
	\begin{lemma}\label{lem:hamiltonian_conformal}
		Let $\psi \colon T^*M \to T^*M$ be a diffeomorphism such that $\psi^*\lambda=a \lambda$ for some $a>0$. 
		Moreover, let $H \colon T^*M \times I \to \bR$ be a compactly supported function and set $\tilde{H} \coloneqq a^{-1}H \circ \psi$. 
		Then, $\psi^{-1} \circ \phi^H_s \circ \psi = \phi^{\tilde{H}}_s$ and $\lambda(X_{\tilde{H}_s}) = a^{-1} \lambda(X_{H_s}) \circ \psi$ for $s \in I$.
	\end{lemma}
	
	\begin{proof}
		For $s \in I$, we have 
		\begin{align*}
			\iota_{X_{\tilde{H}_s}} \omega 
			& = -d (a^{-1} H_s \circ \psi) \\
			& = -a^{-1} \psi^* dH_s \\
			& = a^{-1} \psi^* (\iota_{X_{H_s}}\omega) \\
			& = a^{-1} \iota_{\psi^* X_{H_s}} (\psi^*\omega) \\
			& = \iota_{\psi^* X_{H_s}} \omega,
		\end{align*}
		where $\psi^* X \coloneqq (d\psi)^{-1} X \circ \psi$.
		This shows that $X_{\tilde{H}_s} = \psi^* X_{H_s}$, which implies $\phi^H_s \circ \psi = \psi \circ \phi^{\tilde{H}}_s$.
		Moreover, since $\lambda=a^{-1}\psi^*\lambda$, we have 
		\[
			\iota_{X_{\tilde{H}_s}}\lambda 
			= \iota_{\psi^* X_{H_s}} \lambda 
			= a^{-1} \iota_{\psi^* X_{H_s}}(\psi^*\lambda) 
			= a^{-1} \psi^* \iota_{X_{H_s}}\lambda,
		\]
		which proves the second equality.
	\end{proof}
	
	\begin{lemma}\label{lem:SQ_conformal}
	    For a compactly supported function $H \colon T^*M \times I \to \bR$, let $K_{\phi^H} \in \SD(M^2 \times I \times \bR)$ be the sheaf quantization of the Hamiltonian isotopy $\phi^H$. 
		Then, for such a function $H$, one has $K_{\psi_0^{-1} \circ \phi^H \circ \psi_0} \simeq f_* K_{\phi^H}$, where $f$ is defined as 
		\[
			f \colon M^2 \times \bR \times I \to M^2 \times \bR \times I, \quad  (x_1,x_2,t,s) \mapsto (x_1,x_2,e^{-1}t,s).
		\]
	\end{lemma}
	
	\begin{proof}
		Set $\tilde{H} \coloneqq e^{-1} H \circ \psi_0$. 
		We also define 
		\[
			u_{H,s}(p) \coloneqq \int_0^s (H_{s'}-\lambda(X_{H_{s'}}))(\phi^H_{s'}(p)) \, ds'
		\]
		for $H \colon T^*M \times I \to \bR$ and $s \in I$.
		Then, by \cref{lem:hamiltonian_conformal}, we have
		\begin{align*}
			u_{\tilde{H},s}(p) 
			& = 
			\int_0^s (\tilde{H}_{s'}-\lambda(X_{\tilde{H}_{s'}}))(\psi_0^{-1}\phi^H_{s'} \psi_0(p)) \, ds' \\
			& = 
			e^{-1} \int_0^s (H_{s'}-\lambda(X_{H_{s'}}))(\phi^H_{s'}(\psi_0(p))) \, ds' \\
			& = 
			e^{-1} u_{H,s}(\psi_0(p)).
		\end{align*}
		
		We shall estimate the microsupports of $K_{\psi_0^{-1} \circ \phi^H \circ \psi_0}$ and $f_* K_{\phi^H}$.
		On the one hand, we have 
		\begin{align*}
			& \rMS(K_{\psi_0^{-1} \circ \phi^H \circ \psi_0}) \\
			= {} &  
			\left\{ ((x';\xi'), (x;-\xi), (u_{\tilde{H},s}(x;\xi/\tau),\tau), (s,-\tau \tilde{H}_s(\phi^{\tilde{H}}_s(x;\xi/\tau)))) 
			\; \middle| \; (x';\xi'/\tau) = \phi^{\tilde{H}}_s(x;\xi/\tau) \right\} \\
			= {} & 
			\begin{aligned}
			    & \left\{ ((x';\xi'), (x;-\xi), (e^{-1} u_{H,s}(x;e\xi/\tau),\tau), (s,-\tau e^{-1}H_s(\phi^{H}_s(x;e\xi/\tau)))) \right| \\
			    & \quad \left. (x';e\xi'/\tau) = \phi^{H}_s(x;e\xi/\tau) \right\}. 
			\end{aligned}
		\end{align*}
		On the other hand, we have 
		\begin{align*}
			& \rMS(f_* K_{\phi^H}) \\
			= {} &
            \begin{aligned}
                & \left\{ ((x';\xi'), (x;-\xi), (e^{-1}u_{H,s}(x;\xi/\tau),e\tau), (s,-\tau H_s(\phi^{H}_s(x;\xi/\tau)))) \right| \\
			    & \quad \left. (x';\xi'/\tau) = \phi^{H}_s(x;\xi/\tau) \right\}
            \end{aligned}
			\\
			= {} & 
            \begin{aligned}
                & \left\{ ((x';\xi'), (x;-\xi), (e^{-1}u_{H,s}(x;e\xi/\tilde{\tau}),\tilde{\tau}), (s,-\tilde{\tau} e^{-1} H_s(\phi^{H}_s(x;e\xi/\tilde{\tau})))) \right| \\ 
			    & \quad \left. (x';e\xi'/\tilde{\tau}) = \phi^{H}_s(x;e\xi/\tilde{\tau}) \right\}.
            \end{aligned} 
		\end{align*}
		Since $(f_* K_{\phi^H})|_{s=0} \simeq \bfk_{\Delta \times \{0\}}$, by the uniqueness of the sheaf quantization (\cite{G-K-S}), we conclude.
	\end{proof}
 
\begin{proof}[Proof of \cref{thm:nearbyconj_appendix}]
	By \cref{lem:SQ_conformal}, for $\varphi \in \DHam_c(T^*M)$, we have $\cK_{\psi_0^{-1} \circ \varphi \circ \psi_0} = f_* \cK_\varphi$, where $\cK_\varphi=Q(\varphi)$ and $f$ is defined, by abuse of notation, as
	\[
		f \colon M^2 \times \bR \to M^2 \times \bR, \quad (x_1,x_2,t) \mapsto (x_1,x_2,e^{-1}t).
	\]
	This implies, setting $h=\psi_0^{-1} \circ \psi_1 \in \DHam_c (T^*M)$, that the map 
	\[
		T \colon \DHam_c(T^*M) \to \DHam_c(T^*M), \quad \varphi \mapsto \psi_0^{-1} \circ \varphi \circ \psi_1 = \psi_0^{-1} \circ \varphi \circ \psi_0 \circ h
	\]
	is a contraction.
    Note that $h=\psi_0^{-1} \circ \psi_1 \in \DHam_c (T^*M)$ is proved in the proof of \cite[Theorem~7.4]{AHV2024higher}.
	Indeed, we have 
	\begin{align*}
		d_{\cT(T^*M^2)}(\cK_{T\varphi},\cK_{T\varphi'}) 
		& = 
		d_{\cT(T^*M^2)}(\cK_{\psi_0^{-1} \circ \varphi \circ \psi_0 \circ h}, \cK_{\psi_0^{-1} \circ \varphi' \circ \psi_0 \circ h}) \\
		& = 
		d_{\cT(T^*M^2)}(\cK_{\psi_0^{-1} \circ \varphi \circ \psi_0} \convstar \cK_h, \cK_{\psi_0^{-1} \circ \varphi' \circ \psi_0} \convstar \cK_h) \\
		& = 
		d_{\cT(T^*M^2)}(\cK_{\psi_0^{-1} \circ \varphi \circ \psi_0}, \cK_{\psi_0^{-1} \circ \varphi' \circ \psi_0}) \\
        & = 
        d_{\cT(T^*M^2)}(f_*\cK_\varphi,f_*\cK_{\varphi'}) \\
		& = 
		e^{-1} d_{\cT(T^*M^2)}(\cK_\varphi,\cK_{\varphi'}),
	\end{align*}  
	where the last equality follows from the fact that an $(a,b)$-isomorphism for $(\cK_\varphi,\cK_{\varphi'})$ gives an $(e^{-1}a,e^{-1}b)$-isomorphism for $(f_*\cK_\varphi,f_*\cK_{\varphi'})$ and vice versa.
	
	Hence, the map $T$ extends to the completion as a contraction, which we will also denote by $T \colon \widehat{\DHam_c}(T^*M) \to \widehat{\DHam_c}(T^*M)$. 
	Therefore, there exists a unique fixed point $\varphi_\infty \in \widehat{\DHam_c}(T^*M)$ so that $\psi_0^{-1}\varphi_\infty = \varphi_\infty \psi_1^{-1}$, where both sides should be understood as actions on $\widehat{\frakL}(T^*M)$. 
	Since $L\in \widehat{\mathfrak{L}}(T^*M)$ is the unique fixed point of the action of $\psi_1^{-1}$ and $0_M\in \widehat{\mathfrak{L}}(T^*M)$ is the unique fixed point of the action of $\psi_0^{-1}$, we find $\varphi_\infty(L) = 0_M$. 
 
	By its construction, the element $\varphi_\infty \in \widehat{\DHam_c}(T^*M)$ is represented as the sequence $(\varphi_n)_{n \in \bN}$ with $\varphi_n \coloneqq \psi_0^{-n} \circ \psi_1^n$.
	Thus, we obtain
	\[
		 \liminf_n \varphi_n(L) =
		 \liminf_n \psi_0^{-n} \circ \psi_1^n(L) = \liminf_n \psi_0^{-n}(L) = 0_M.
	\]
    Since taking the microsupport is ``continuous'' by~\cite[Proposition~6.26]{GV2022singular}, for any $F \in \cT_L(T^*M)$, we have 
	\[
		\RS(K_{\varphi_\infty}^\convstar(F)) 
		\subset \liminf_n \varphi_n(L) = 0_M,
	\]
	which proves the result.   
\end{proof}

\printbibliography

\end{document}